\documentclass [12pt,english] {article}
\usepackage{amsthm}
\usepackage{amsmath}
\usepackage{amsfonts}
\usepackage{amssymb}
\usepackage[all]{xy}
\usepackage{color}   
\usepackage{tikz} \usetikzlibrary{arrows}

\usepackage{listings}
\usepackage{setspace}
\usepackage[font=small,labelfont=bf]{caption}
\singlespace

\usepackage[utf8]{inputenc}

\newtheorem{defi}{Definition}
\newtheorem{teo}{Theorem}
\newtheorem{prop}{Proposition}
\newtheorem{lema}{Lemma}
\newtheorem{coro}{Corollary}

\newtheorem{obser}{Remark}
\usepackage[textwidth=2.5cm, textsize=small, colorinlistoftodos]
{todonotes}

\begin{document}
\title{Kernels in digraphs with colored vertices}

\author{Mucuy-kak Guevara\footnote{Universidad Nacional Aut\'onoma de M\'exico, mucuy-kak.guevara@ciencias.unam.mx}, Teresa I. Hoekstra-Mendoza\footnote{Centro de Investigaci\'on en Matem\'aticas, maria.idskjen@cimat.mx}, Miguel Licona-Velazquez\footnote{Universidad Nacional Aut\'onoma de M\'exico, eliconav23@xanum.uam.mx}}

\date{\empty}

\maketitle

\begin{abstract}
    In this paper, we introduce the concept of up-color kernel, which is a generalization of a kernel for vertex-colored digraphs. We give sufficient and necessary conditions for several families of digraphs to have an up-color kernel, as well as for certain products of digraphs.
\end{abstract}
\section{Introduction}
For general concepts, we refer the reader to \cite{CL}. In graph theory, an undirected graph coloring is an assignment of labels, which can be colors, to the vertices, the edges, or both sets of a graph. The first results about graph coloring start with the four-color conjecture, postulated by Francis Guthrie in 1852, which states that exactly four colors are sufficient to color a map so that no regions sharing a common border are colored with the same color. 
Another widely studied topic is the domination problem in undirected graphs. For example, the book Fundamentals of Domination in Graphs by Haynes et al. \cite{HHS} contains a bibliography with over 1200 papers on the subject. On the other hand, the Roman domination \cite{C}, \cite{F}, \cite{Y} describes a function that labels the set of vertices in a graph with the values 0, 1, or 2, the function is used to determine the minimum weight of a labeling satisfying that every vertex $u$ for which $f(u)=0$ is adjacent to at least one vertex $v$ for which $f(v)=2$. Also we have other kind of domination with colors, the Rainbow domination \cite{BHR}, \cite{HQ}, \cite{Y}  requires that each vertex is dominated by at least one vertex of each color. Inspired by this kind of concepts, in \cite{GGMR} define the Up-color domination, which deals with a function $c$ that labels the set of vertices in a graph with values $0, 1, 2,\ldots,$ such that the colors given to two adjacent vertices are different and $D$ a subset of vertices of the graph that is an up-color dominant which means that if for any vertex $v\notin D$ there is $d,$ a neighbor of $v$, such that $d\in D$ and $c(v)<c(d)$ and no vertex in $D$ is labeled with 0. While Roman domination gives a coloring of the graph with three colors, Rainbow domination deals with $k$ colors for $k\geq 1$ and in Up-color domination there exists a hierarchy in the vertices of the graph determined by the colors, and each vertex must be dominated by a higher colored vertex. 

Let $D=\{V(D),A(D)\}$ be a digraph where $V(D)$ denotes the set of vertices of $D$ and $A(D)$ is the set of arcs of $D$. A subset $N$ of $V(D)$ is said to be a \textit{kernel} if it is both independent (there are no arcs between any pair of vertices in $N)$ and absorbent (for all $u\in V(D)\setminus N$ there exists $v$ in $N$ such that $(u,v)\in A(D)).$ The concept of kernel was introduced in \cite{VNM} by von Newmann and Morgenstern in the context of Game Theory as a solution for cooperative $n$-players games.   Chvátal proved in \cite{CV} that determining if a digraph has a kernel is an NP-complete problem. Kernels have been studied by several authors, see for example, \cite{BG}, \cite{HV}, \cite{Matus} and \cite{RV}.
While there are many generalizations of kernel for colored digraphs, see for example \cite{SSW} and \cite{AL}, all of them are for digraphs colored by arcs. In this work with inspiration by the concepts mentioned above, mainly up-color domination and kernel, we introduce the concept of up-color kernel as follows.
 
\begin{defi}
Let $D$ be a digraph. $D$ is a $c$-colored digraph if $V(D)$ is colored with the numbers in $\{0,1,\ldots\}.$ We say that a function $c:V(D)\rightarrow \{0,1,\ldots\}$ is a $c$-coloring of $D$.
 \end{defi}

\begin{defi}\label{upkernel}
Let $D$ be a $c$-colored digraph. We say that a set $N\subset V(D)$ is an up-color absorbent set if the following two conditions hold.
\begin{enumerate}
    \item For every vertex $v\notin N$ there exists a vertex $w\in N$ with $(v,w)\in A(D)$ and $c(v)<c(w),$ and
    \item $N$ does not contain any vertex having color zero.
\end{enumerate}
 \end{defi}

 \begin{defi}
Let $D$ be a $c$-colored digraph and  $K\subset V(D)$. We say that $K$ is an up-color kernel of $D$ if $K$ is up-color absorbent and for  $v,w\in N,$  $(v,w),(w,v)\notin A(D),$ i.e $K$ is an independent set.
 \end{defi}

The concept of up-color kernel is closely related to that of a kernel in vertex-colored graphs. In particular, for a digraph to possess an up-color kernel, it must first have a kernel. Moreover, if a digraph $D$ has a unique kernel $K$, then for any $c$-coloring  of $D$ where a vertex of $K$ is assigned color zero, the digraph does not admit an up-color kernel.  Additionally, the up-color kernel is linked to the concept of Up-color domination, defined for undirected graphs, since any vertex $v\notin K$ there is a vertex $x\in K$ such that $(v,x)$ is an arc and $c(v)<c(x).$ 

Richardson theorem \cite{R} states that every digraph without a directed odd cycle has a kernel. A natural question to ask is: Does every $c$-colored digraph without a directed odd cycle have an up-color kernel? The answer is negative, see Figure \ref{Psin}, where we have a $c$-colored path such that $v,$ the vertex of out-degree zero, is colored with color $0,$ however $v$ must belong to an up-color kernel of the path but by definition this could not happen. Now we can ask: Which $c$-colored digraphs always have an up-color kernel? 

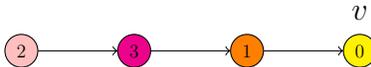
\begin{figure}[h!]
    \centering
    \begin{tikzpicture}
        \node[circle, draw, fill=pink, scale=.6] (0) at (0,0){2};
        \node[circle, draw, fill=magenta, scale=.6] (1) at (1.5,0){3};
        \node[circle, draw, fill=orange, scale=.6] (2) at (3,0){1};
        \node[circle, draw, fill=yellow, scale=.6] (3) at (4.5,0){0};
        \node (v) at (4.5,0.5){$v$};
        \draw[->] (0)--(1);
        \draw[->] (1)--(2);
        \draw[->] (2)--(3);
    \end{tikzpicture}
    \caption{$c$-colored path with out an up-color kernel.}
    \label{Psin}
\end{figure}
In Section \ref{Sec1} we work with $c$-colored digraphs having a simple structure like paths, cycles, trees, wheels, to name a few. For this kind of digraphs we give sufficient and necessary conditions to guarantee the existence of an up-color kernel.

Another line of investigation is, given a digraph with an up-color kernel: Which operations between digraphs have an up-color kernel?

A digraph (graph) product \cite{HIK} constitutes a binary operation on digraphs. It involves taking two or more digraphs, denoted as $D_{1} = (V(D_{1}),A(D_{1}))$ and $D_2 = (V(D_{2}),A(D_{2}))$ and combining them to form a new graph $D'$ with specific characteristics. 
In Section \ref{CSProduct}  we work with the Cartesian and Strong products of paths,  stars and cycles. For the  graphs obtained through  these operations, we provide necessary and sufficient conditions on the coloring to ensure the existence of an up-color kernel.

In Section \ref{ZykovS} we work with a well-known operation called Zykov sum \cite{Z}, defined as follows, let $G$ be a connected digraph and let $\mathcal{H}=\{H_v\}_{v\in V(G)}$ be a family of pairwise disjoint digraphs. The Zykov sum of $G$ and $\mathcal{H}$, denoted by $G[\mathcal{H}]$, is the digraph obtained from $G$ by replacing every vertex $v$ of $G$ with the digraph $H_v$ and there is an arc from every vertex of $H_u$ to any vertex of $H_v$ if $(u,v)\in A(G).$ For the digraph $G[\mathcal{H}]$  we give necessary and sufficient conditions on the family  $\mathcal{H}=\{H_v\}_{v\in V(G)}$  and the coloring in order to guarantee the existence of an up-color kernel in $G[\mathcal{H}]$, when the graph $G$ is either a path of a cycle. In the case where $G$ is an odd cycle, we observe that $G[\mathcal{H}]$ does not have up-color kernel, regardless of the family $\mathcal{H}=\{H_v\}_{v\in V(G)}$ and the coloring of the graphs.
  
In Section \ref{inex-crown}, inspired by the definition of generalized crown for edge-colored digraphs given in \cite{MV}, we define the following two operations. The generalized in-crown, denoted by $D\triangledown  \mathcal{H}$ and the generalized ex-crown, denoted by $D\vartriangle \mathcal{H}$ where $D$ is a $c$-colored digraph, $\mathcal{H}=\{H_{i}\}_{i\in I}$ is a sequence of pairwise disjoint $c$-colored digraphs.
The in-crown and ex-crown of $D$ and $\mathcal{H}$ are obtained from $D\sqcup \bigsqcup_{i\in I}H_i$ by adding arcs between $D$ and $\mathcal{H}$.
In both operations we give necessary and sufficient conditions on $D$, $\mathcal{H}=\{H_{i}\}_{i\in I}$ and the coloring to guarantee the existence of an up-color kernel. 

In Section \ref{L(D)} we work with a well-known operation that involves obtaining a new digraph called the line digraph from a $c$-colored digraph $D.$ The line digraph $L(D)$ with the outer coloration, has the arcs of $D$ as its vertices.  There is an arc from a vertex $x$ to a vertex $y$ if the final vertex of $x$ is the same as the initial vertex of $y$ in $D.$ We color these vertices with the color of their  final vertices in $D.$ We proved that, under a certain condition on the colors of the vertices of in-degree equal to zero and its neighbors, the number of up-color kernels in $D$  is equal to the number of up-color kernels in $L(D)$ with the outer coloration. Finally we also define the inner coloration for $L(D)$ which consists in coloring each vertex in $L(D)$  with the color of its initial vertex in $D$. We provide an example of digraph such that its line digraph inner colored does not preserve up-color kernels.
 
\section{Up-color kernels in digraphs}\label{Sec1}

 \subsection{Paths, cycles and trees}
 \begin{prop}
Let $G_{n}$ denote a directed path  with $V(G_{n})=\{x_0, \dots, x_{n}\}$, $A(G_{n})=\{(x_i,x_{i-1}):1\leq i\leq n-1\}$ 
and $c:V(G_{n})\rightarrow \{0,1,\ldots\}$ a $c$-coloring of $G_{n}$. Then $G_{n}$ has an up-color kernel if and only if $c(x_{2i})>c(x_{2i+1})$ for every $ i \geq 0.$ \end{prop}

\begin{proof}
    If  the above mentioned condition holds, then the set $K=\{x_0, x_{2}, \dots\}$ is an up-color kernel for $G_n.$ Now assume that $G_n$ has an up-color kernel $K$. Clearly $c(x_0)>c(x_{1})$ since otherwise either $x_0$ or $x_{1}$ is not up-color absorbed. Assume that for some $i$ we have $c(x_{2i})\leq c(x_{2i+1}).$ Then since $x_{2i+1}$ is not up-color absorbed by any vertex, we must have $x_{2i+1}\in  K$ for every up-color kernel $K.$ This means that $x_j \in K$ for every odd index $j$ but this contradicts the fact that $x_0\in K.$
\end{proof}
\begin{coro}
    Let $G$ be a directed cycle of even length and vertex set $\{x_1, \dots, x_n.\}$ Then $G$ has a up-color kernel if and only if one of the following holds:
    \begin{enumerate}
        \item $c(x_{2i})> c(x_{2i-1})$  for every $2i \in \mathbb{Z}_n$ or
       \item $c(x_{2i+1})> c(x_{2i})$ for every $2i \in \mathbb{Z}_n$. 
    \end{enumerate}
\end{coro}

\begin{defi}
    Given a forest $B,$ we can subdivide its vertices in levels by considering the vertices with out-degree zero to have level zero, their in-neighborhoods which do not have level zero, to have level one and so on. If the subforest induced by the vertices of an even level has arcs, we can again subdivide it into sublevels. We continue this way until every even sublevel consists of isolated vertices.
    We say that a vertex is \textbf{even-leveled} if all of its sublevels are even and \textbf{odd-leveled} otherwise. 
\end{defi}
\begin{obser}
    The set of even vertices is an independent set, as we can see for example in Figure \ref{par}.
\end{obser}

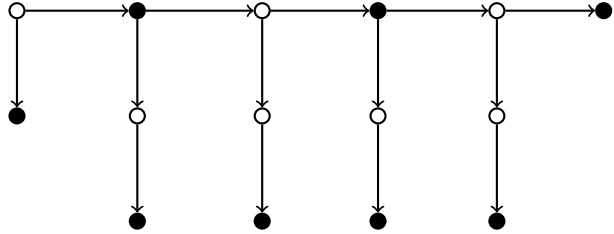
\begin{figure}[h!]
\begin{center}
 \begin{tikzpicture}[scale=2,thick]
		\tikzstyle{every node}=[minimum width=0pt, inner sep=2pt, circle]
			\draw (-3.5,2.0) node[draw] (0) {};
            \draw (-3.5,1.3) node[draw, fill=black] (10) {};
			\draw (-2.7,2.0) node[draw,  fill=black] (1) {};
			\draw (-1.87,2.0) node[draw] (2) {};
			\draw (-1.1,2.0) node[draw,  fill=black] (3) {};
			\draw (-0.31,2.0) node[draw] (4) {};
			\draw (-2.7,1.3) node[draw] (5) {};
			\draw (-1.87,1.3) node[draw] (6) {};
			\draw (-1.1,1.3) node[draw] (7) {};
			\draw (-0.31,1.3) node[draw] (8) {};
            \draw (0.4, 2) node[draw, fill=black] (9) {};
			\draw (-2.7,0.6) node[draw, fill=black] (50) {};
			\draw (-1.87,0.6) node[draw, fill=black] (60) {};
			\draw (-1.1,0.6) node[draw, fill=black] (70) {};
			\draw (-0.31,0.6) node[draw, fill=black] (80) {};
\draw[->] (4) edge (9);
\draw[->] (0) edge (10);
			\draw[->]  (0) edge (1);
			\draw[->]  (1) edge (2);
			\draw[->]  (2) edge (3);
			\draw[->]  (3) edge (4);
			\draw[->]  (1) edge (5);
			\draw[->]  (2) edge (6);
			\draw[->]  (3) edge (7);
			\draw[->]  (4) edge (8);
            \draw[->]  (5) edge (50);
			\draw[->]  (6) edge (60);
			\draw[->]  (7) edge (70);
			\draw[->]  (8) edge (80);
		\end{tikzpicture}
  \caption{The black vertices correspond to the even vertices.} \label{par}
	\end{center}
\end{figure}

\begin{teo}
    A forest $B$ has up-color kernel if and only if for every odd-leveled vertex $v$ there exists an even-leveled ex-neighbor which has greater color than $v.$
\end{teo}
\begin{proof}
    Assume that for every odd vertex $v$ there exists an even ex-neighbor which has greater color than $v.$ Then the set of even-leveled vertices is an up-color kernel for $B.$ Now assume that $B$ contains an up-color kernel $K.$ Notice that $K$ must contain every  even-leveled vertex. This means that $K$ does not contain any odd-leveled vertex and hence every odd-leveled vertex $v$ there exists an even-leveled ex-neighbor which has greater color than $v.$
\end{proof}

\subsection{Wheels, unicyclic graphs and cycles with chords}
\begin{prop}
    Let $W_n$ denote the orientation of the wheel graph containing a directed cycle of length $n.$ Then $W_n$ has an up-color kernel if and only if one of the following conditions hold:
    \begin{enumerate}
        \item The unique vertex $v$ with $d(v)=n$ absorbs every other vertex and has the greatest color, or
        \item $n$ is even, $v$ is up-color absorbed by some vertex $w$, and $w$ belongs to an up-color kernel $K$ of $W_n\setminus \{v\}.$
    \end{enumerate}
    \end{prop}
    \begin{proof}
        Assume one of the items holds. If the first item holds, then $v$ is an up-color kernel for $W_n.$ If the second item holds, the set $K$ is an up-color kernel for $W_n.$

        Now assume that $W_n$ has an up-color kernel. Since $v$ is adjacent to every other vertex, either $v$ is the up-color kernel, or $v$ is up-color absorbed by some vertex $w.$ If $v$ is the up-color kernel then $v$ up-color-absorbs every other vertex and $v$ has the greatest color, thus the first item holds. If $v$ is up-color absorbed by $w$, then $W_n\setminus \{v\}$ must have an up-color kernel containing $w$ and since $W_n\setminus \{v\}$ is a directed cycle, $n$ must be even.
    \end{proof}
\begin{prop}\label{cicllopelo}
Let $G$ denote a graph obtained by adding an outward pending arc $(w,v)$ to a directed cycle $C$. If $C$ is of odd length then $G$ has an up-color kernel if and only if $c(v)>c(w)$ and $C \setminus \{w\}=G \setminus \{v,w\}$ has a up-color kernel.
\end{prop}
\begin{proof}
    If $C\setminus \{w\}$ has a up-color kernel $K$, then $K\cup \{v\}$ is an up-color kernel for $G.$
    Now assume that either $c(v)\leq c(w)$ or $C\setminus \{w\}$ does not have an up-color kernel. Notice first that since $d^+(v)=0$, $v$ must belong to every up-color kernel of $G. $ If $c(v)\leq c(w)$ then $v$ does not up-color-absorbs any vertex of $C$ and thus an up-color kernel of $G$ would contain an up-color kernel of $C$ which is impossible. Now assume that $C\setminus \{w\}$ does not contain an up-color kernel. Since $w$ can not belong to a kernel of $G,$ if $K$ is an up-color kernel of $G$, then $K\setminus \{v\}$ is an up-color kernel of $C\setminus \{w\}.$
\end{proof}

\begin{prop}
Let $H$ be a graph without an up-color kernel, $I$ a maximum independent set and $J=\{w_1,\ldots, w_r\}$ the set of vertices which are not up-color absorbed by $I.$  Define $G$ a graph obtained by adding outward pending arcs $(w_i,v)_i$ to $H$. Then $G$ has an up-color kernel if and only if $c(v_i)>c(w_i)$ and $H \setminus J =G \setminus (\bigcup\limits_{i=1}^r\{v_i,w_i\})$ has a up-color kernel.
\end{prop}
\begin{proof}
    The proof is analogous to the proof of Proposition \ref{cicllopelo}.
\end{proof}

\begin{coro}
Let $H$ be a graph with an up-color kernel and define $G$  a graph obtained by adding an outward pending arc $(w,v)$ to $H$. Then $G$ has an up-color kernel if and only if $c(v)>c(w)$ and $H \setminus \{w\}=G \setminus \{v,w\}$ has a up-color kernel or $0< c(v) \leq c(w)$  and $H$ has an up-color kernel which does not contain $w$.
\end{coro}

\begin{coro}
Let $H$ be a graph with an up-color kernel and define $G$ a graph obtained by adding outward pending arcs, $(w_1,v_1), \ldots (w_l, v_l)$, to  $H$. Then $G$ has an up-color kernel if and only if $c(v_i)>c(w_i)$ for all $i=1\ldots l$ and $H \setminus \{w_1\ldots, w_l\}=G \setminus \{v_i,w_i\}_{i=1,\ldots, l}$ has a up-color kernel or $0< c(v_j) \leq c(w_j)$ for some $j$ and $H$ has an up-color kernel which does not contain $w_j$, for those $j$.
\end{coro}

\begin{teo}
Let $G$ be an odd cycle with one chord (which we allow to be a double arc). Then $G$ has an up-color kernel if and only if there exists a vertex $v\in V(G)$ which up-color-absorbs two vertices $u,w$ and  the subgraph $G\setminus\{u,v,w\}$ which is a path or the union of two paths, has an up-color kernel.
\end{teo}
\begin{proof}
  If $G\setminus \{u,v,w\}$ has an up-color kernel $K$ then $K\cup \{v\}$ is an up-color kernel for $G.$
  Now assume that $G$ has an up-color kernel. If every vertex of $K$ absorbs only one vertex, we can either find a kernel for an odd cycle which is impossible, or there exists an even cycle $C$ which has an up-color kernel $K_c$ such that $G\setminus C=P$ is a path of odd length having a kernel $K_p$ with $K=K_p\cup K_c.$  Let $P=\{x_1,\dots, x_{2k+1}\}$ and notice that there exist vertices $c_1, c_n\in V(C)$ such that $(c_1,x_1), (x_n,c_n), (c_1, c_n)\in A(G).$ Since $P$ is a path, $x_n\in K_p$ and since $P$ is of odd length, $x_1\in K_p$ but this means that $c_1, c_n\notin K_c$ which is a contradiction.
  This means that there exists $v\in V(G)$ which up-color-absorbs two vertices $u$ and $w$ and since $v$ is the only vertex of $G$ having in-degree two, $v$ must belong to every up-color kernel of $G.$ Since $ u,w\notin K$, $K\setminus\{v\}$ is an up-color kernel for $G\setminus\{u,v,w\}.$
\end{proof}

\begin{lema}
A complete digraph always has an up-color kernel and a tournament has an up-color kernel if and only if the vertex with greatest color has out-degree zero.
\end{lema}
\begin{proof}
    In both cases, an up-color kernel must consist of a unique vertex. In a complete digraph, every vertex is absorbent, thus the vertex with the greatest color is an up-color kernel. In a tournament, only vertices with out-degree zero are absorbent thus a tournament has an up-color kernel if and only if the vertex with the greatest color has out-degree zero.
\end{proof}

\section{Products}
\subsection{Cartesian and strong products}\label{CSProduct}
\begin{defi}
 Let $D_{1}=(V(D_{1}),A(D_{1}))$ and $D_{2}=(V(D_{2}),A(D_{2}))$ two digraphs. Their Cartesian product, denoted as $D_{1}\square D_{2}$, forms a digraph with a vertex set given by $V(D_{1}\square D_{2})= V(D_{1})\times V(D_{2})$ and an arc set denoted by $A(D_{1}\square D_{2})$. An element $((v_{1}, v'_{1}),(v_{2}, v'_{2}))$ belongs to $A(D_{1}\square D_{2})$ if and only if
 \begin{enumerate}
     \item $v_{1}= v_{2}$, and there is an arc $(v'_{1},v'_{2})\in A(D_{2})$ or
     \item $v'_{1}= v'_{2}$, and there is an arc $(v_{1},v_{2})\in A(D_{1}).$
 \end{enumerate}
\end{defi}

\begin{defi}
 Let $D_{1}=(V(D_{1}),A(D_{1}))$ and $D_{2}=(V(D_{2}),A(D_{2}))$ two digraphs. Their strong product, denoted as $D_{1}\boxtimes D_{2}$, forms a digraph with a vertex set given by $V(D_{1}\boxtimes D_{2})= V(D_{1})\times V(D_{2})$ and an arc set given by $A(D_{1}\boxtimes D_{2})=A(D_1\square D_2)\cup \{((x,y), (z,w)): (x,z)\in A(D_1), (y,w)\in A(D_2)\}$.
\end{defi}

\begin{teo}\label{cubo}
    Let $(G,c)$ with $G=\gamma_1 \square \dots \square \gamma_n$ be such that each $\gamma_i$ is a directed path for $1 \leq i \leq n$. Denote by $D_0$ the unique vertex in $G$ with $\delta^+(D_0)=0$ and let $D_1=N^-(D_0)$, $D_2=N^-(N^-(D_0))$ and in general $D_i=N_i^-(D_0).$ Then $G$ has an up-color kernel if and only if for every vertex $v\in D_{2i+1}$  there exists $w \in D_{2i}$ such that $(v,w)\in A(G)$ and $c(v)<c(w)$. 
\end{teo}
\begin{proof}
    We are going to prove first that $G$ has an up-color kernel. Notice first that each $D_i$ is an independent set.
    Let $K=D_0\cup D_2\cup \dots \cup D_{2m}$  where $D_{2m}$ is either the unique vertex with in-degree zero or its out-neighborhood. Then clearly $K$ is up-color absorbent and since every arc in $G$ goes from a vertex in $D_i$ to a vertex in $D_{i-1}$ for some $ 1 \leq i \leq m$ we have that $K$ is also an independent set. Thus $K$ is an up-color kernel.

    Now assume that $G$ has an up-color kernel $K. $ Then $D_0 \in K$ thus $D_1 \cap K=\emptyset$ and hence for  every $v \in D_1$ we have $c(v)<c(D_0).$ Since $D_1\cap K=\emptyset$ and  $N^+(D_2)=D_1$ we have that $D_2 \subset K$ and thus $D_3 \cap K =\emptyset.$ Continuing this way we have that $D_{2i}\subset K$ and $D_{2i+1}\cap K=\emptyset$ for every $1 \leq i \leq m. $ Since we are assuming that $K$ is an up-color kernel for $G$ we have that for every vertex $v\in D_{2i+1}$  there exists $w \in D_{2i}$ such that $(v,w)\in A(G)$ and $c(v)<c(w)$ for every $1 \leq i \leq m. $
 \end{proof}

\begin{figure}[h!]
    \centering
    \includegraphics[scale=.5]{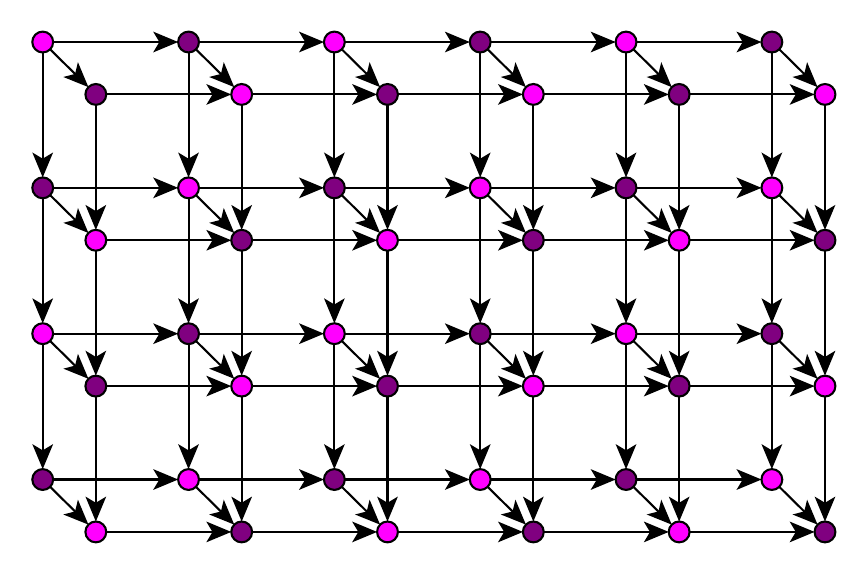}
    \caption{The purple vertices represent an up-color kernel for $G$ if the conditions of Theorem \ref{cubo} hold.}
    \label{cub}
\end{figure}

\begin{teo}\label{fuerte}
    Let $(G,c)$ with $G=\gamma_1 \boxtimes \dots \boxtimes \gamma_n$ be such that each $\gamma_i$ is a directed path for $1 \leq i \leq n$. Denote by $D_0$ the unique vertex in $G$ with $\delta^+(D_0)=0$ and let $D_1=N^-(D_0)$, $D_2=N^-(N^-(D_0))$ and in general $D_i=N_i^-(D_0).$ Denote by $B_i$ a maximal independent set of $D_i.$
    
    Then $G$ has an up-color kernel if and only if for every vertex $v\in D_{2i+1}$  there exists $w \in B_{2i}$ such that $(v,w)\in A(G)$ and $c(v)<c(w)$. 
\end{teo}
\begin{proof}
    The proof is analogous to the proof of Theorem \ref{cubo}.
\end{proof}

\begin{figure}[h!]
    \centering
    \includegraphics[scale=.5]{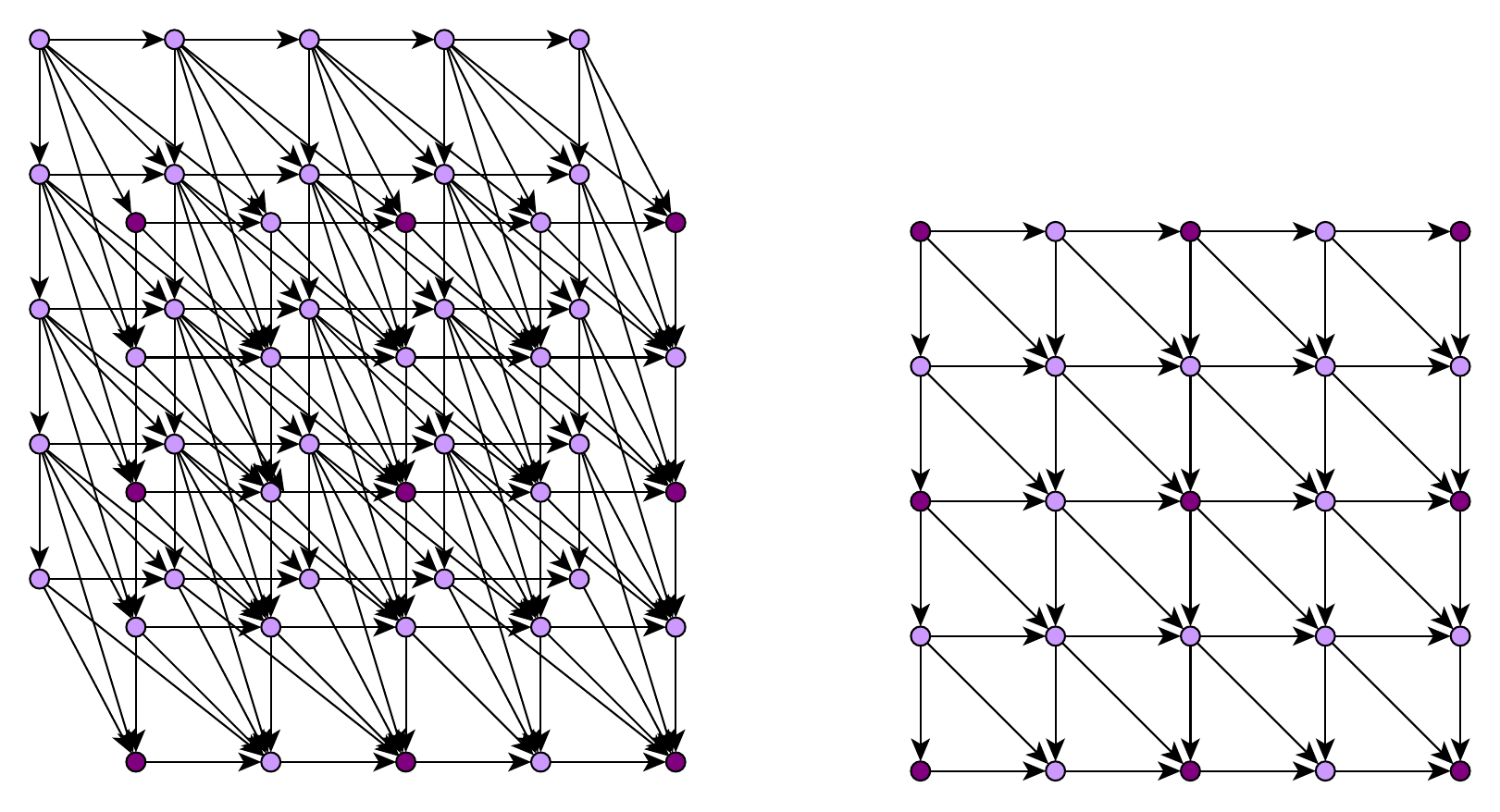}
    \caption{The dark purple vertices represent an up-color kernel for $G$ if the conditions of Theorem \ref{fuerte} hold.}
    \label{cubf}
\end{figure}

Let $S_k^+$  and $S_k^-$ denote the orientation of the star graph $K_{1,k}$ which have one unique vertex with zero in-degree and zero out-degree, respectively.
In what follows we will assume that the central vertex of any star is labeled with the number zero and the remaining vertices are labeled with the numbers $1,2,\dots$

\begin{teo}\label{yx}
    Let $G=S_{m_1}^-\square \cdots \square S_{m_p}^-\square S_{n_1}^+\square \cdots \square S_{n_q}^+$ be a colored digraph. Define a function $f:V(G)\rightarrow \mathbb{N}\times \mathbb{N}$ as $f(v)=(y,x)$, if there are $y$ non-zero entries among the first $p$ entries of $v$ and $x$ zero entries among the last $q$ entries of $v.$ Then $G$ has an up-color kernel if and only if the set $I=\{v\in V(G): x+y= 2k \text{ for some }k\in \mathbb{Z} \text{ where } f(v)=(x,y) \}$ is up-color absorbent.
\end{teo}
\begin{proof}
Define the sets $V_{i,j}=\{v\in V(G): f(v)=(i,j)\}$.  Every arc in $G$  goes from $V_{i,j}$ to $V_{i-1,j}$ or from $V_{i,j}$ to  $V_{i,j-1}$ for some $1\leq i\leq p$ and some $1\leq j \leq q.$ Each vertex in $V_{i,j}$ has at least one out-neighbor in $V_{i-1,j}$ and at least one out-neighbor in $V_{i,j-1}$. This means that the set $V_{0,0}$ consists precisely of the vertices having zero out-degree. 

Thus, if $I=\bigcup_{i+j=2k}V_{i,j}$ is an up-color absorbent set, it is an up-color kernel for $G$ since it is an independent set.

Now assume that $G$ has an up-color kernel $K.$ Then $K$ contains the set $V_{0,0}$, and $K$ cannot contain vertices of the set $V_{1,0}\cup V_{1,0}$. This means that $K$ must contain every vertex of $V_{0,2}\cup V_{2,0}\cup V_{1,1}$ and $V_{0,2}\cup V_{2,0}\cup V_{1,1}$ must up-color absorb the set $V_{1,0}\cup V_{1,0}$. This implies that $K$
 cannot contain any vertex of the set $V_{1,2}\cup V_{2,1}.$ Continuing this way, we obtain that $K$ must contain every vertex of $I,$ cannot contain any vertex of $V(G)\setminus I$ and hence $K=I$ is an up-color absorbing set.

\end{proof}

\begin{figure}[h!]
    \centering
    \includegraphics[scale=.4]{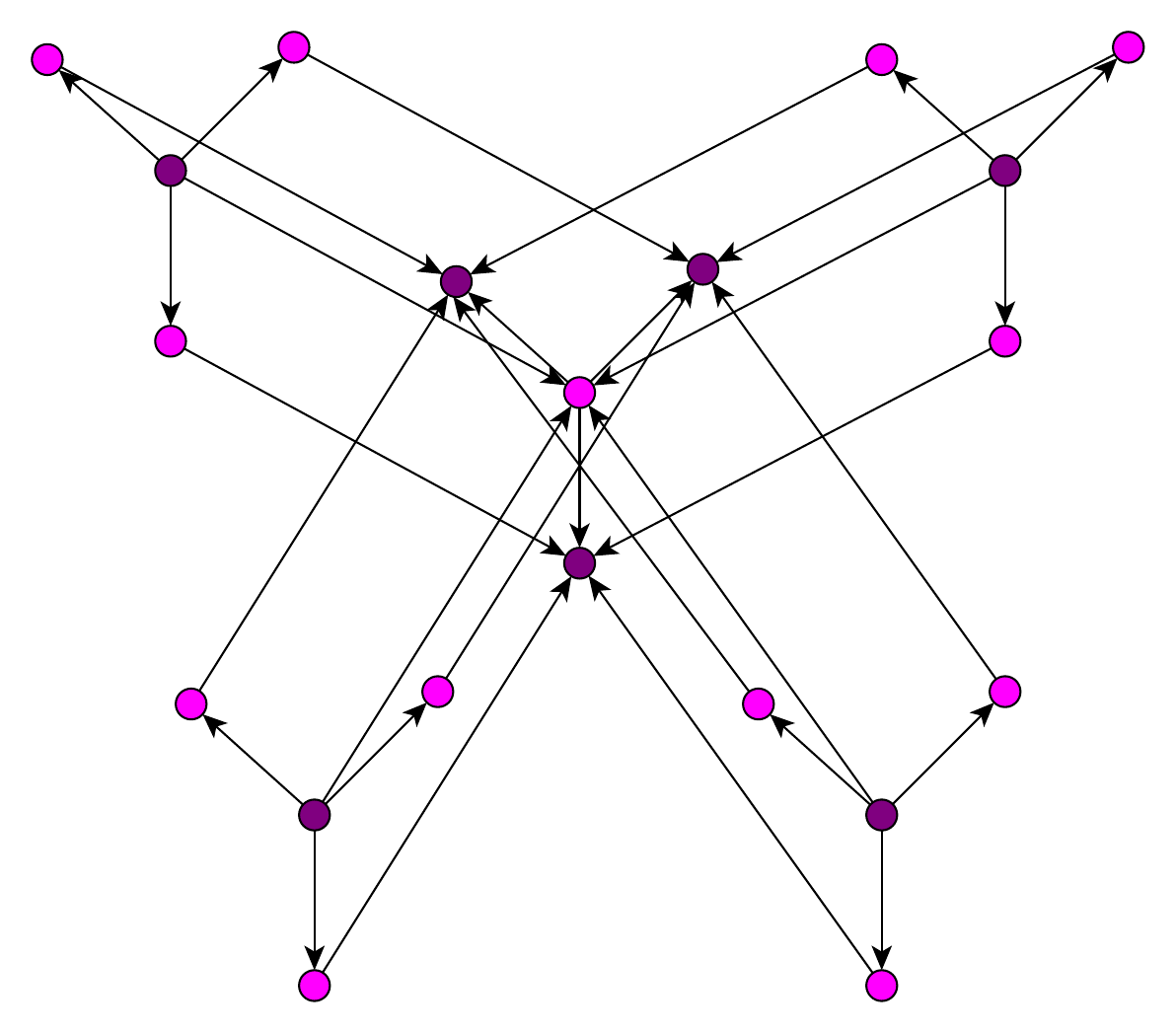}
    \caption{The purple vertices form an up-color kernel for $S_3^+\square S_4^-$ if the conditions of Theorem \ref{yx} hold.}
    \label{YX}
\end{figure}

\begin{teo}
   Let $G=S_{m_1}^-\boxtimes \cdots \boxtimes
   S_{m_p}^-\boxtimes S_{n_1}^+\boxtimes \cdots \boxtimes S_{n_q}^+$ be a colored digraph. Then $G$ has an up-color kernel if and only if the set $\{(0,\dots,0, x_1, \dots,x_q)\in V(G): x_i\neq 0\}=\{v\in V(G): d^+(v)=0\}$ is an up-color absorbing set.
\end{teo}
\begin{proof}
    Notice that, for any coloring of $G$, the set $\{v\in V(G): d^+(v)=0\}$ is a kernel for $G$ since it is independent and absorbent. Thus $\{v\in V(G): d^+(v)=0\}$ is an up-color kernel if and only if it is up-color absorbent.
\end{proof}

Denote by $\overrightarrow{K}_{m,n}$ the orientation of the complete bipartite graph where every arc is oriented from a vertex in the set of $m$ vertices to a vertex in the set of $n$ vertices.

\begin{teo}\label{starpathin}
     Let $G=P_k\square \overrightarrow{K}_{m,n}$ with $V(\overrightarrow{K}_{m,n})=X\cup Y$ with $|X|=m $, $|Y|=n$, 
      $V(P_k)=\{1, \dots, k\}$ and $A(P_k)=\{(i,i-1)\}_{i=2}^k.$ Then $G$ has an up-color kernel if and only if for every $x\in\{x_1, \dots, x_m\}$, $c(2i,x)<c(2i-1,x)$ and for every $y \in \{y_1, \dots y_n\}$, either $c(2i+1,y)<c(2i,y)$, or $c(2i-1,y)< c(2i-1,x)$ for some $x \in  \{x_1, \dots, x_m\}$ for $1\leq i \leq \frac{k}{2}$. 
\end{teo}
\begin{proof}
    Assume the conditions of the theorem hold. Then the set $\{(2i-1,x): 1\leq i \leq \frac{k}{2}, x\in X\}\cup\{(2i,y):1\leq i \leq \frac{k}{2}, y\in Y\}$ is an up-color kernel for $G.$

    Now assume that $G$ has an up-color kernel $K.$ Notice that since $d^+((1,x)=0$ we have $(1,x)\in K$ for every $x\in X$. This means that $(2,x), (1,y)\notin K$ for $ y\in Y$ and hence $(2,x)$ must be up-color absorbed by $(1,x)$, and $(1,y)$ must be up-color absorbed by $(1,x)$ for some $x\in X.$ We also have that $(2,y), (3,x)\in K$ since $N^+((2,y))=\{(1,y)\}\cup \bigcup_{x\in X} \{(2,x)\}$ and $N^+((3,x))=\{(2,x)\}.$ This means that $(3,y)\notin K$ and thus either  $c(3,y)<c(2,y)$ or $ c(3,y)< c(3,x)$  for some $x\in X.$ Continuing this way we obtain that $(2i,y), (2i-1,x)\in K$ for  $1\leq i \leq \frac{k}{2}$, thus $(2i-1,y),(2i,x)\notin K$ and hence the vertices $(2i-1,y),(2i,x)$ must be up-color absorbed for every $x\in X, y\in Y.$
\end{proof}

\begin{figure}[h!]
    \centering
    \includegraphics[scale=.6]{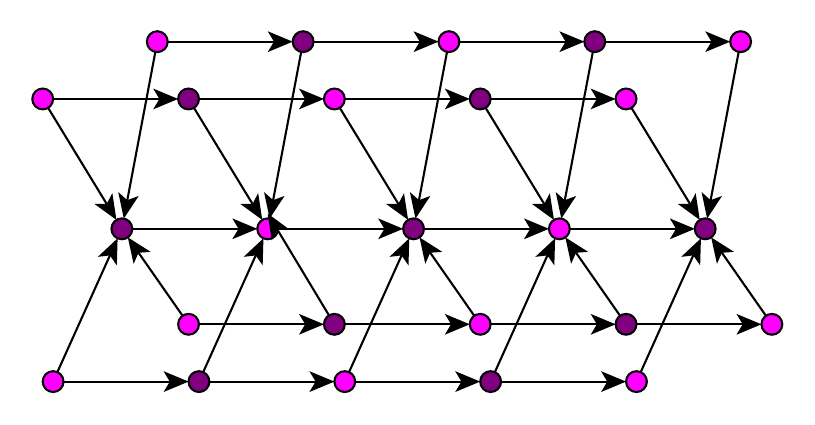}
    \caption{The graph $P_5\square S_4^-$, where the purple vertices are an up-color kernel if the conditions of Theorem \ref{starpathin} hold.}
    \label{S4P5}
\end{figure}

\begin{prop}\label{prodcic} 
    Let $D=C_{2k}\square C_{2m}$ and let $V(C_{2k})=\{x_1, \dots, x_{2k}\}$ and $V(C_{2m})=\{y_1, \dots, y_{2m}\}$. Then $D$ has an up-color kernel if and only if one of the following holds:
    \begin{itemize}
    \item  Every vertex $(x_{2i}, y_{2j})$ with $ 1 \leq i \leq k $, $1 \leq j \leq m$ and $(x_{2i+1}, y_{2j+1})$ with $ 1 \leq i \leq k-1 $, $1 \leq j \leq m-1$ has an out-neighbor with greater color or, 
    \item  Every vertex $(x_{2i+1}, y_{2j})$ with $ 1 \leq i \leq k -1$, $1 \leq j \leq m$ and $(x_{2i}, y_{2j+1})$ with $ 1 \leq i \leq k $, $1 \leq j \leq m-1$ has an out-neighbor with greater color.
        \end{itemize}
\end{prop}
\begin{proof}
    Assume one of the items holds. Then either $$\{(x_{2i+1}, y_{2j+1}):  1 \leq i \leq k-1,1 \leq j \leq m-1\} \cup\{(x_{2i}, y_{2j}):  1 \leq i \leq k,1 \leq j \leq m\}\text{ or}$$  $$\{(x_{2i}, y_{2j+1}):  1 \leq i \leq k,1 \leq j \leq m-1\} \cup\{(x_{2i+1}, y_{2j}):  1 \leq i \leq k-1,1 \leq j \leq m\}$$ is an up-color kernel for $D.$
    Now assume that $D$ has an up-color kernel $K.$ If $(x_1,y_1)\in K$ then $(x_1, y_{2m}), (x_{2k},y_1)\notin K$ thus 
    $(x_{2k}, y_{2m})\in K.$ Continuing this way we have that $$(x_{2k}-i(\text{mod }2k ), y_{2m}-i(\text{mod } 2m)) 
    \in K$$ for $ 1\leq i \leq 2km-2$, in other words, $K$ contains the set $\{(x_{2i+1}, y_{2j+1}):  1 \leq i \leq k-1,1 
    \leq j \leq m-1\} \cup\{(x_{2i}, y_{2j}):  1 \leq i \leq k,1 \leq j \leq m\}$, thus $K$ cannot contain the set $
    \{(x_{2i}, y_{2j+1}):  1 \leq i \leq k,1 \leq j \leq m-1\} \cup\{(x_{2i+1}, y_{2j}):  1 \leq i \leq k-1,1 \leq j \leq 
    m\}$ and hence the vertices $(x_{2i+1}, y_{2j})$ and $(x_{2i}, y_{2j+1})$ have an out-neighbor with greater color 
    for every $ 1 \leq i \leq k$ and every $1 \leq j \leq m$ thus the second item holds. If we were to assume that $K$ 
    does not contain the vertex $(x_1, y_1)$ but the vertex $(x_1, y_2)$ or $(x_2,y_1)$, we can analogously proof that the first item holds.
\end{proof}

\begin{figure}[h!]
    \centering
    \includegraphics[scale=.4]{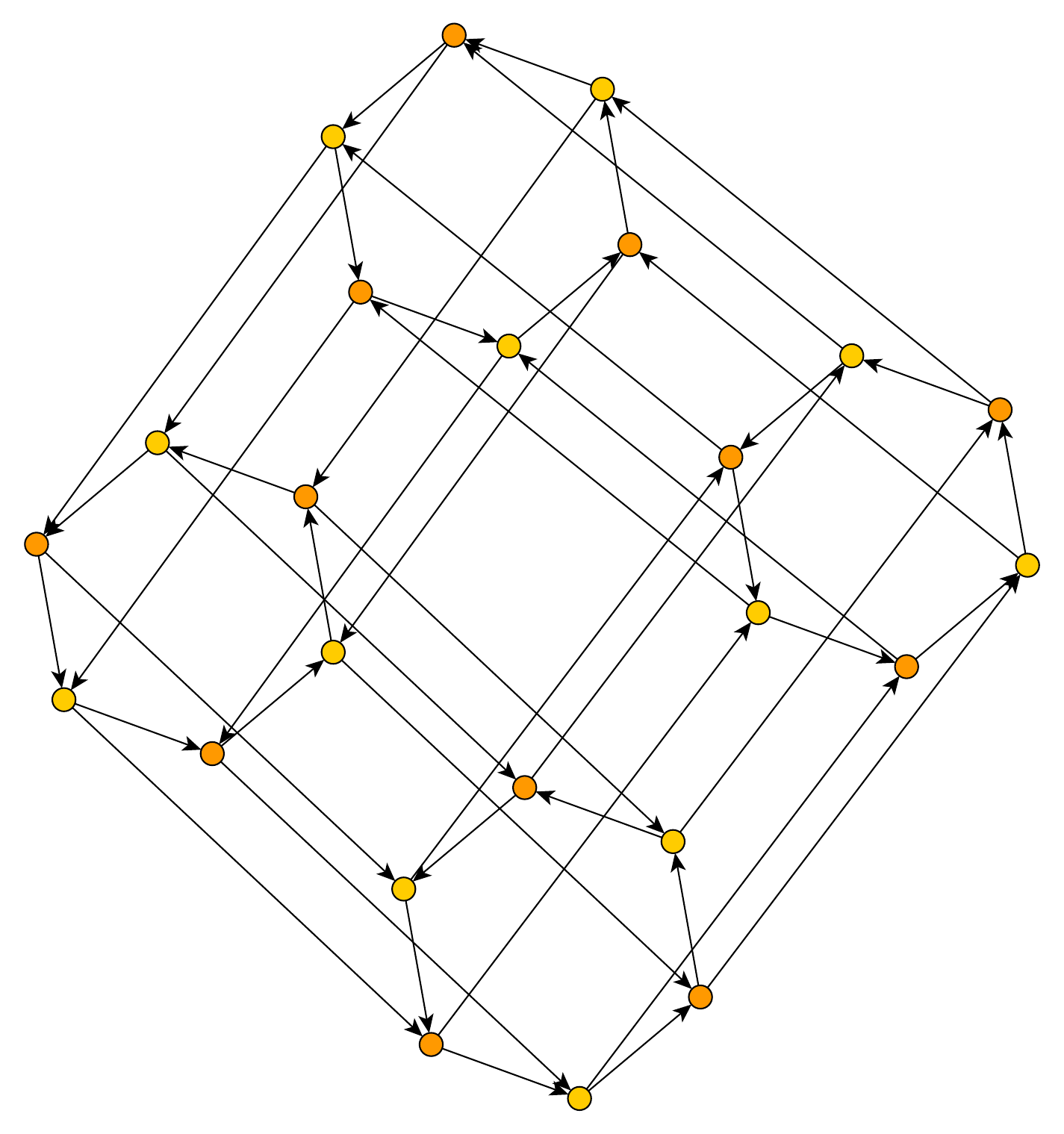}
    \caption{The orange or the yellow vertices form an up-color kernel for $C_4\square C_6$ if the conditions of Proposition \ref{prodcic} hold.}
    \label{C4sqC6}
\end{figure}

\subsection{Zykov sum}\label{ZykovS}

\begin{defi} 
    Let $G$ be a connected digraph and let $\{\mathcal{H}_v\}_{v\in V(G)}$ denote a family of pairwise disjoint digraphs. The Zykov sum of $G$ and $\mathcal{H}$, denoted by $G[\mathcal{H}]$, is the digraph obtained from $G$ by replacing every vertex $v$ of $G$ with the digraph $H_v$ and there is an arc from every vertex of $H_u$ to any vertex of  $H_v$ if $(u,v)\in A(G).$ 
\end{defi}

\begin{prop}
    Let $P_n$ be the directed path graph with $V(P_n)=\{1, \dots, n\}$ and arcs $(i+1,i)$ for $ 1\leq i \leq n-1$. Let $\mathcal{H}=\{G_i\}_{i=1}^n$ be a family of $n$ pairwise disjoint digraphs. Then the Zykov sum $P_n(\mathcal{H})$ has an up-color kernel if and only if every digraph $G_{2i-1}$ has an up-color kernel $K_{2i-1}$ which up-color-absorbs every vertex of $G_{2i}$ for $ 1\leq i \leq \frac{n}{2}.$
\end{prop}
\begin{proof}
    Assume first that $G(\mathcal{H})$ has an up-color kernel $K$.  Then, $K\cap V(G_1)=K_1$ is an up-color kernel for $G_1.$ Since for every vertex $v\in V(G_2)$ there exists an arc $(v,w)$ for every $w\in V(G_1)$, every vertex in $V(G_2)$ must be up-color absorbed by some vertex in $K_1.$ Thus $K\cap V(G_2)=\emptyset$ and hence $K\cap V(G_3)=K_3$ must be an up-color kernel for $G_3.$ Since for every vertex $v\in V(G_4)$ there exists an arc $(v,w)$ for every $w\in V(G_3)$, every vertex in $V(G_4)$ must be up-color absorbed by some vertex in $K_3.$ Continuing this way, we have that every digraph $G_{2i-1}$ has an up-color kernel $K_{2_i-1}$ which up-color-absorbs every vertex of $G_{2i}$ for $ 1\leq i \leq \frac{n}{2}.$

    Now assume that every digraph $G_{2i-1}$ has an up-color kernel $K_{2_i-1}$ which up-color-absorbs every vertex of $G_{2i}$ for $ 1\leq i \leq \frac{n}{2}.$ Then clearly $K_1 \cup K_3 \cup \dots \cup K_{m} $ is an up-color kernel for $P_n(\mathcal{H})$ for $m = n$ of $n$ is odd and $m=n-1$ if $n$ is even.
\end{proof}

\begin{coro}
    Let $C_{2n}$ be the directed cycle graph with $V(C_{2n})=\{1, \dots, 2n\}$, and let $\mathcal{H}=\{G_i\}_{i=1}^{2n}$ be a family of $n$ pairwise disjoint digraphs. Then the Zykov sum $C_{2n}(\mathcal{H})$ has an up-color kernel if and only if every digraph $G_{2i-1}$ has an up-color kernel which up-color absorbs every vertex of $G_{2i}$, or every digraph $G_{2i}$ has an up-color kernel which absorbs every vertex of $G_{2i+1}.$
\end{coro}

\begin{coro}
    Let $C_{2n+1}$ be the directed cycle graph and $\mathcal{H}$ a family of $2n+1$ graphs. Then the Zykov sum $C_{2n+1}(\mathcal{H})$ does not have an up-color kernel.
\end{coro}

\subsection{Up-color kernels in $D$, in $D\triangledown  \mathcal{H}$ and in $D\vartriangle \mathcal{H}$.}\label{inex-crown}

Inspired by undergraduate thesis \cite{MV}, where several generalizations of the corona operation are presented, we give the following definition:

\begin{defi}\label{incr}
     Let $D$  be a digraph, 
     $\mathcal{H}= (H_{i})_{i\in I}$ be a sequence of pairwise disjoint  digraphs with $V(H_{i})=\{ y_{1}^{i},\ldots, y_{p_{i}}^{i}\}$ and $p_{i} \geq 2$ for each $i\in I$. The generalized in-crown and ex-crown of the digraph $D$ and the sequence $\mathcal{H}$ are the digraphs $D\triangledown  \mathcal{H}$ and $D\triangle  \mathcal{H}$ respectively such that:
     \begin{itemize}
         \item $V(D\triangledown  \mathcal{H})= V(D)\bigcup\limits_{i\in I} V(H_{i}) = V(D\triangle  \mathcal{H})$
         \item $A(D\triangledown  \mathcal{H})= A(D)\bigcup\limits_{i\in I} A(H_{i})\bigcup\limits_{i\in I} \{ (y_{t}^{i},x)\},$ 
         \item  $A(D\triangle  \mathcal{H})= A(D)\bigcup\limits_{i\in I} A(H_{i})\bigcup\limits_{i\in I} \{(x,y_{t}^{i})\},$
      \end{itemize} with $t\in \{1\ldots,p_{i}\}$ for some $x\in V(D). $ 
\end{defi}

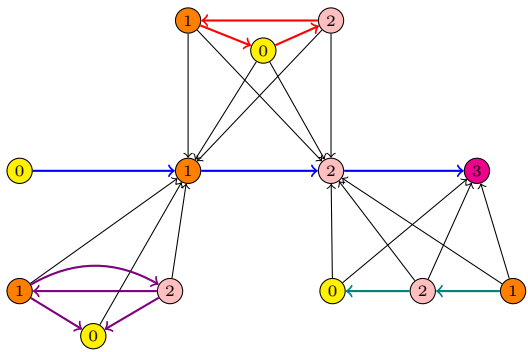
\begin{figure}[h!]
    \centering

    \begin{tikzpicture}[scale=2]
		\tikzstyle{every node}=[minimum width=0pt, inner sep=1.5pt, circle]
			\draw (-3.12,0.3) node[draw, fill=yellow] (0) { \tiny 0};
			\draw (-2.0,0.3) node[draw, fill=orange] (1) { \tiny 1};
			\draw (-1.05,0.3) node[draw, fill=pink] (2) { \tiny 2};
			\draw (-0.08,0.3) node[draw, fill=magenta] (3) { \tiny 3};
			\draw (-2,1.3) node[draw, fill=orange] (4) { \tiny 1};
			\draw (-1.05,1.3) node[draw, fill=pink] (5) { \tiny 2};
			\draw (-1.5,1.1) node[draw, fill=yellow] (6) { \tiny 0};
			\draw (0.16,-0.5) node[draw, fill=orange] (7) { \tiny 1};
			\draw (-0.44,-0.5) node[draw, fill=pink] (8) { \tiny 2};
			\draw (-1.04,-0.5) node[draw, fill=yellow] (9) { \tiny 0};
			\draw (-2.12,-0.5) node[draw, fill=pink] (10) { \tiny 2};
			\draw (-3.12,-0.5) node[draw, fill=orange] (11) { \tiny 1};
			\draw (-2.63,-0.8) node[draw, fill=yellow] (12) { \tiny 0};
			\draw[->, blue, thick]  (0) edge (1);
			\draw[->, blue, thick]  (1) edge (2);
			\draw[->, blue, thick]  (2) edge (3);
			\draw[->, red, thick]  (6) edge (5);
			\draw[->, red, thick]  (4) edge (6);
            \draw[->, red, thick] (5) edge (4);
			\draw[->]  (5) edge (2);
			\draw[->]  (6) edge (2);
			\draw[->]  (4) edge (2);
			\draw[->] (4) edge (1);
			\draw[->] (6) edge (1);
			\draw[->] (5) edge (1);
			\draw[->, teal, thick] (7) edge (8);
			\draw[->, teal, thick] (8) edge (9);
			\draw[->] (7) edge (3);
			\draw[->] (8) edge (3);
			\draw[->] (9) edge (3);
			\draw[->] (7) edge (2);
			\draw[->] (8) edge (2);
			\draw[->] (9) edge (2);
			\draw[->, violet, thick] (10) edge (11);
			\draw[->, violet, thick]  (11) edge (12);
			\draw[->, violet, thick]  (10) edge (12);
			\draw[->]  (10) edge (1);
			\draw[->]  (11) edge (1);
			\draw[->]  (12) edge (1);
            \draw[->, bend left, violet, thick] (11) edge (10);
		\end{tikzpicture}
    \caption{Example of an in-crown of $D$ with $\mathcal{H}$, where $D$ is the path on four vertices whose arcs are blue.}
    \label{incrown}
\end{figure}

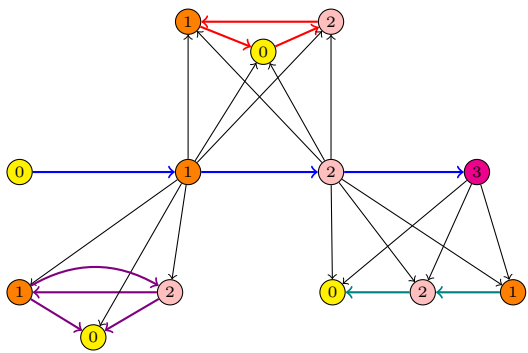
\begin{figure}[h!]
    \centering
    \begin{tikzpicture}[scale=2]
		\tikzstyle{every node}=[minimum width=0pt, inner sep=1.5pt, circle]
			\draw (-3.12,0.3) node[draw, fill=yellow] (0) { \tiny 0};
			\draw (-2.0,0.3) node[draw, fill=orange] (1) { \tiny 1};
			\draw (-1.05,0.3) node[draw, fill=pink] (2) { \tiny 2};
			\draw (-0.08,0.3) node[draw, fill=magenta] (3) { \tiny 3};
			\draw (-2,1.3) node[draw, fill=orange] (4) { \tiny 1};
			\draw (-1.05,1.3) node[draw, fill=pink] (5) { \tiny 2};
			\draw (-1.5,1.1) node[draw, fill=yellow] (6) { \tiny 0};
			\draw (0.16,-0.5) node[draw, fill=orange] (7) { \tiny 1};
			\draw (-0.44,-0.5) node[draw, fill=pink] (8) { \tiny 2};
			\draw (-1.04,-0.5) node[draw, fill=yellow] (9) { \tiny 0};
			\draw (-2.12,-0.5) node[draw, fill=pink] (10) { \tiny 2};
			\draw (-3.12,-0.5) node[draw, fill=orange] (11) { \tiny 1};
			\draw (-2.63,-0.8) node[draw, fill=yellow] (12) { \tiny 0};
			\draw[->, blue, thick]  (0) edge (1);
			\draw[->, blue, thick]  (1) edge (2);
			\draw[->, blue, thick]  (2) edge (3);
			\draw[->, red, thick]  (6) edge (5);
			\draw[->, red, thick]  (4) edge (6);
            \draw[->, red, thick] (5) edge (4);
			\draw[->]  (2) edge (5);
			\draw[->]  (2) edge (6);
			\draw[->]  (2) edge (4);
			\draw [->] (1) edge (4);
			\draw[->]  (1) edge (6);
			\draw[->] (1) edge (5);
			\draw[->, teal, thick] (7) edge (8);
			\draw[->, teal, thick] (8) edge (9);
			\draw[->] (3) edge (7);
			\draw[->] (3) edge (8);
			\draw[->] (3) edge (9);
			\draw[->] (2) edge (7);
			\draw[->] (2) edge (8);
			\draw[->] (2) edge (9);
			\draw[->, violet, thick] (10) edge (11);
			\draw[->, violet, thick]  (11) edge (12);
			\draw[->, violet, thick]  (10) edge (12);
			\draw[->]  (1) edge (11);
			\draw[->]  (1) edge (10);
			\draw[->]  (1) edge (12);
            \draw[->, bend left, violet, thick] (11) edge (10);
		\end{tikzpicture}
    \caption{Example of an ex-crown of $D$ with $\mathcal{H}$, where $D$ is the path on four vertices whose arcs are blue.}
    \label{excrown}
\end{figure}

In the theorems \ref{in} and \ref{ex} we will assume that $\mathcal{H}=(H_{i})_{i\in I}$ holds the condition described in the  Definition \ref{incr}.
\begin{teo}\label{in}
    Let $D$ a $c$-colored digraph and $\mathcal{H}=(H_{i})_{i\in I}$  be a sequence of pairwise disjoint $c$-colored digraphs and $D \triangledown \mathcal{H}$ the generalized in-crown. Then  $D \triangledown \mathcal{H}$ has an up-color kernel if and only if there exists an up-color kernel $K$ for $D$ such that for all $(y,x)\in A(D\triangledown  \mathcal{H})$ such that $x\in K$ and $y\in H_{i},$ $c(y)<c(x)$ and, if there are no arcs from $H_i$ to $K$, then $H_i$ has an up-color kernel $N_i.$
      
\end{teo}

\begin{proof}
    Assume the conditions hold for $D$ and $\mathcal{H}.$ Then clearly $K \cup (\bigcup_{i\in I}N_i)$ is an up-color kernel for $D \triangledown \mathcal{H}$. Now assume that $D \triangledown \mathcal{H}$ has an up-color kernel $K$. Since $D$ is a terminal component of $D \triangledown \mathcal{H}$, $K \cap D$ must be a kernel for $D.$ Let $(x,y)$ be an arc for $ x\in H_i$ and $y\in K.$ Since $K$ is an up-color kernel, $x\notin K$ and thus $ c(x) <c(y).$ 
    Finally, assume that $H_i$ is such that there are no arcs from $H_i$ to $K$ for some $i\in I.$ Then $K\cap H_i$ must be an up-color kernel for $H_i.$
\end{proof}

\begin{teo}\label{ex}
    Let $D$ a $c$-colored digraph and $\mathcal{H}=(H_{i})_{i\in I}$  be a sequence of pairwise disjoint $c$-colored digraphs and $D \triangle \mathcal{H}$ the generalized ex-crown. Then $D \triangle \mathcal{H}$ has an up-color kernel if and only if the following conditions hold:
    \begin{itemize}
        \item for all $i\in I,$ $H_{i}$ has up-color kernel $N_{i}$,
        \item given $B=\{x\in V(D): \text{ there is no } (x,H_{i})\in A(D \triangle \mathcal{H})\},$ the subdigraph induced by $B$ has up-color kernel $K.$
        \item for all $(x,y)\in D \triangle \mathcal{H}$ such that $y\in K \cup N_i$, there exists $z\in K \cup N_i$ such that $(x,z)\in A (D \triangle \mathcal{H})$ with $c(x)<c(z).$
    \end{itemize}
\end{teo}

\begin{proof}
    Assume all three items hold for $D$ and $\mathcal{H}.$ Then clearly $K \bigcup_{i\in I}N_i$ is an up-color kernel for $D \triangle \mathcal{H}$. Now assume that $D \triangle \mathcal{H}$ has an up-color kernel $K.$ Since each $H_i$ is a terminal component, $K \cap H_i$ must be an up-color kernel for $H_i$ for every $i\in I.$ Let $(x,y)$ be an arc such that $y\in D$ and $x \in K$. Then, since $y \notin x$, there exists a vertex $z\in K$ such that $c(y)<c(z).$ Finally, notice that since $B \subset D$, we have that $B \cap K $ must be an up-color kernel for $\langle B \rangle.$
\end{proof}

\subsection{Up-color kernels in $D$ and in $L(D)$}\label{L(D)}

\begin{defi}
The line digraph of $D=(V(D),A(D))$ is the digraph denoted by $L(D)$ with its set of vertices as the set of arcs of $D$, and for any $h,k\in A(D)$ there is $(h, k)\in A(L(D)$  if and only if the corresponding arcs $h,k$ induce a directed path in $D$; this is, the terminal endpoint of $h$ is the initial endpoint of $k$.
\end{defi}
In what follows, we denote the arc $h=(u, v)\in A(D)$ and the vertex $h$ in $L(D)$ by the same symbol. 

\begin{defi}\label{colout}
Let $D$ be a $c$-colored digraph and $L(D)$ its line digraph, the \textbf{outer coloration} of $L(D),$ is the vertex-coloration  defined as:  $c_{_{L(D)}}((u,v)))=c_{_{D}}(v).$
\end{defi}

\begin{defi}\label{def1}
Let $D$ be a digraph, we denote by $\mathcal{P}(V(D))$ the power set of $V(D)$ and by $\mathcal{P}(V(L(D)))$ the power set of  $V(L(D))$. Let us define  $f:\mathcal{P}(V(D))\longrightarrow \mathcal{P}(V(L(D)))$ as follows:

For every $Z\subseteq V(D)$, $f(Z)=\{(u,x)\in V(L(D)):x\in Z\}=A(\langle N^-[Z]\rangle).$ 
\end{defi}
The following Lemma was proved by  Galeana-Sánchez and Pastrana-Ramírez in \cite{GP} and will be useful in this section.
\begin{lema}\cite{GP}\label{Zind}
Let $D$ a digraph, then $Z\subseteq V(D)$ is independent in $D$ if and only if $f(Z)$ is independent in %the outer coloration of 
$L(D).$ 
\end{lema}
\begin{proof}
Let $D$ be a digraph and $Z$ an independent set. Then $\langle N^-[Z]\rangle$ is a union of (not necessarily disjoint) inward directed stars thus the longest directed path in $\langle N^-[Z]\rangle$ has length two and hence $f(Z)$ is independent in $L(D).$ 

Now assume that $f(Z)$ is independent, thus there is no directed path of length three in $\langle N^-[Z]\rangle.$ This means that when taking the endpoints of the arcs in $f(Z)$, we obtain an independent set.

\end{proof}

\begin{teo}\label{numeroupkernels}
Let $D$ a $c$-colored digraph where for all $u,v\in V(D)$ such that $\delta^{-}(u)=0$ and $v\in N^{+}(u),$ $c_{D}(u)<c_{D}(v)$. Then the number of up-color kernels in $D$ is equal to the number of up-color kernels in the outer coloration of $L(D).$        
\end{teo}
\begin{proof}
Let us denote by $K$ the set of all of the up-color kernel in $D$ and  by $K^{\ast}$ the set of all of the up-color kernels in the outer coloration of $L(D).$

\

\textit{Claim 1.} If $Z\in K$ then $f(Z)\in K^{\ast}.$  

Since $Z\in K$,   $Z$  is an independent set, and  by Lemma \ref{Zind}  $f(Z)$ is also an independent set thus we only need to show that $f(Z)$ is up-color absorbent in $L(D).$ 
Let $a=(x,y)$  in $V(L(D))\setminus f(Z),$ thus $y\notin Z.$ Since $Z\in K$, there exists $z\in Z$  such that $b=(y,z)\in A(D)$ and  $c_{_{D}}(y)<c_{_{D}}(z).$ This means that $(a,b)=((x,y),(y,z))\in A(L(D))$ with $c_{_{L(D)}}(a)<c_{_{L(D)}}(b).$  Hence $f(Z)$ is an up-color absorbent set and   $f(Z)\in K^{\ast}.$  

\

\textit{Claim 2.} The function $f'=f|_K:K\longrightarrow K^{\ast}$ is an injection.

Let $Z_{1},Z_{2}\in K$ be such that  $Z_{1}\neq Z_{2}.$ We may assume that there exists $u\in Z_1$ such that $u\notin Z_2.$
Since $Z_{2}$ is an up-color kernel, there exists $v\in Z_{2}$ such that $b=(u,v)\in A(L(D)).$ It follows from Definition \ref{def1} that $b\in f'(Z_{2}).$ Since $u\in Z_{1}$ and $Z_{1}$ is an independent set we have that $v\notin Z_{1},$ thus $b\notin f'(Z_{1}).$ Therefore $b\in f'(Z_{2})\setminus f'(Z_{1})$ and $f'(Z_{1})\neq f'(Z_{2}).$ 

\

Let us define $g:\mathcal{P}(V(L(D)))\longrightarrow \mathcal{P}(V(D))$ as $g(H)=f^{-1}(H)\bigcup D(H)$ where $D(H):= $
$$\{x\in V(D):\delta_{D}^{-}(x)=0 \text{ and there is no arcs between $x$ and the elements in $f^{-1}(H)$}\}.$$

\textit{Claim 3.} If $H\in K^{\ast}$ then $g(H)\in K.$ 
\\
 If $H$ is an independent set in the outer coloration of $L(D)$ then $g(H)$ is an independent set in $D$ since, by Lemma \ref{Zind}, $f^{-1}(H)$ is independent and by definition $D(H)\cup f^{-1}(H)$ is independent.
\\

We are going to prove that $g(H)$ is an up-color absorbent set in $D$ when $H \in K^{\ast}.$ 
 
 Let $u\in V(D)\setminus g(H).$ Since $u\notin (f^{-1}(H) \cup D(H))$, we have that there is no arc in $H$ whose terminal endpoint is $u$, and at
least one of the two following conditions holds: % or .
\begin{enumerate}
    \item  $\delta^{-}_{D}(u)>0$ thus there is $a=(x,u)\in V(L(D))\setminus H.$ Then there  exists $b=(u,v)$ such that $v\in g(H),$  $(a,b)\in A(L(D))$ and $c_{_{L(D)}}(a)<c_{_{L(D)}}(b).$ Therefore  $c_{_{D}}(u)<c_{_{D}}(v)$ with $v \in g(H).$ 
    \item $\delta^{-}_{D}(u)=0$ and there exists an arc between $u$ and an element in $f^-(H)$, this is, there exists $(u,v)\in A(D)$ for some $v\in f^-(H).$ Since $\delta^{-}_{D}(u)=0$ we must have that $c_{_{D}}(u)<c_{_{D}}(v)$.
\end{enumerate}
Hence, $g(H)$ is an up-color absorbing set.

\

\textit{Claim 4.} The function $g':K^{\ast}\longrightarrow K$  where $g'$ is the restriction of $g$ to $K^{\ast}$ is an injection.

Let $Z_{1},Z_{2}\in K^{\ast}$ such that  $Z_{1}\neq Z_{2}.$ Suppose that $ Z_{1}\setminus Z_{2}\neq \emptyset,$ so let $a=(u,v)\in Z_{1}\setminus Z_{2},$ by the definition of $g'$ we have that $v\in g'(Z_{1})$. Now we will prove that $v\notin g'(Z_{2}).$   Since $Z_{2}$ is an up-color absorbent and $a\notin Z_{2},$ there exists $b\in Z_{2}$ such that $(a,b)\in A(L(D))$ and $c_{_{L(D)}}(a)<c_{_{L(D)}}(b).$  By the definition of $L(D)$ $b=(v,w).$  Notice that $w\in g'(Z_{2}).$  By Claim 3.1 $g'(Z_{2})$ is an independent set, therefore $v\notin g'(Z_{2}).$  Thus $v\in g'(Z_{1})\setminus g'(Z_{2}) $  and $g'(Z_{1})\neq g'(Z_{2}).$

It follows from claims \textit{2} and \textit{4} that $|K|<|K^{\ast}|<|K|.$ Hence $|K|=|K^{\ast}|$
\end{proof}

Figure \ref{contraejemplo1} give us an example why the hypothesis that for all $u,v\in V(D)$ such that $\delta^{-}(u)=0$ and $v\in N^{+}(u),$ $c_{D}(u)<c_{D}(v)$  necessary in Theorem \ref{numeroupkernels}. Observe that the set of the blue vertices in $L(D)$ is an up-color kernel but $D$ has no up-color kernels.

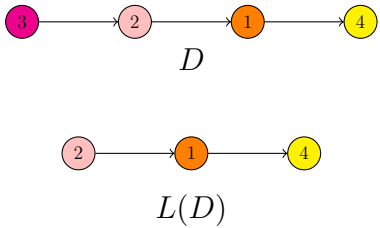
\begin{figure}[h!]
    \centering
\begin{tikzpicture}
        \node[circle, draw, fill=magenta, scale=.6] (0) at (0,0){3};
        \node[circle, draw, fill=pink, scale=.6] (1) at (1.5,0){2};
        \node[circle, draw, fill=orange, scale=.6] (2) at (3,0){1};
        \node[circle, draw, fill=yellow, scale=.6] (3) at (4.5,0){4};
        \node (v) at (2.25,-0.5){$D$};
        \draw[->] (0)--(1);
        \draw[->] (1)--(2);
        \draw[->] (2)--(3);

        \node[circle, draw, fill=pink, scale=.6](1x) at (0.75,-1.75){2};
        \node[circle, draw, fill=orange, scale=.6] (2x) at (2.25,-1.75){1};
        \node[circle, draw, fill=yellow, scale=.6] (3x) at (3.75,-1.75){4};
        \draw[->] (1x)--(2x);
        \draw[->] (2x)--(3x);
        \node (lv) at (2.25,-2.5){$L(D)$};
\end{tikzpicture}
    
    \caption{$L(D)$ having an up-color kernel and $D$ without up-color kernel.}
    \label{contraejemplo1}
\end{figure}

\begin{defi}\label{colin}
Let $D$ be a $c$-colored digraph and $L(D)$ its line digraph, the \textbf{inner coloration} of $L(D),$ is the vertex-coloration  defined as: if $h$ is an arc in $D$ such that $v$ is its initial endpoint then $c_{_{L(D)}}(h)=c_{_{D}}(v).$
\end{defi}

Definition \ref{colin} is an other inherent coloring for $L(D)$ induced by the coloring of $D$, thus a natural question is if the inner coloration works as an hypothesis for Theorem \ref{numeroupkernels}. However, Figure \ref{contraejemplo2} shows that the set of blue vertices is an up-color kernel in $D$ and it is easy to verify that $L(D)$ does not have an up-color kernel.  
\begin{figure}[h!]
    \centering
\begin{tikzpicture}
        \node[circle, draw, fill=pink, scale=.6] (0) at (0,0){2};
        \node[circle, draw, fill=magenta, scale=.6] (1) at (1.5,0){3};
        \node[circle, draw, fill=orange, scale=.6] (2) at (3,0){1};
        \node[circle, draw, fill=yellow, scale=.6] (3) at (4.5,0){4};
        \node (v) at (2.25,-0.5){$D$};
        \draw[->] (0)--(1);
        \draw[->] (1)--(2);
        \draw[->] (2)--(3);

        \node[circle, draw, fill=pink, scale=.6] (0x) at (0.75,-1.75){2};
        \node[circle, draw, fill=magenta, scale=.6](1x) at (2.25,-1.75){3};
        \node[circle, draw, fill=orange, scale=.6] (2x) at (3.75,-1.75){1};
        \draw[->] (0x)--(1x);
        \draw[->] (1x)--(2x);
        \node (lv) at (2.25,-2.25){$L(D)$};
\end{tikzpicture}
    
    \caption{$D$ having an up-color kernel and $L(D)$ inner colored without up-color kernel.}
    \label{contraejemplo2}
\end{figure}
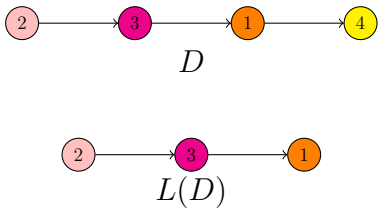

\end{document}